\documentclass[11pt,a4paper, xcolor]{article}
\usepackage{amsfonts}

\usepackage{amssymb,amsthm, amsmath,graphicx,bm,ebezier, mathtools}

\usepackage{hyperref}
\usepackage{color,enumerate}
\usepackage{url,lineno}
\usepackage{arydshln}
\usepackage{lscape}

\usepackage[ruled,vlined]{algorithm2e}

\newtheorem{theorem}{Theorem}
\newtheorem{definition}[theorem]{Definition}
\newtheorem{proposition}[theorem]{Proposition}
\newtheorem{corollary}[theorem]{Corollary}
\newtheorem{lemma}[theorem]{Lemma}

\theoremstyle{remark}

\newtheorem{example}[theorem]{Example}
\newtheorem{remark}[theorem]{Remark}

\usepackage{tikz}
\newcommand{\Cross}{\mathbin{\tikz [x=1.45ex,y=1.45ex,line width=.1ex] \draw (0,0) -- (1,1) (0,1) -- (1,0);}}%
\def\lcm{\mathrm{lcm}}
\def\CH{\mbox{\tt ConvexHull}}
\def\VSet{\mbox{\tt VSet}}

\def\CaS{\mathcal{S}}

\def\N{\mathbb{N}}
\def\Z{\mathbb{Z}}
\def\R{\mathbb{R}}
\def\Q{\mathbb{Q}}

\def\T{\mathrm{T}}
\def\Pr{\mathrm{P}}
\def\Ba{\mathrm{B}}
\def\EH{\mathrm{EH}}

\def\P{\mathbf{P}}
\def\I{\mathbf{I}}

\title{A geometrical characterization of proportionally modular affine semigroups}
\author{
J. D. D\'{\i}az-Ram\'{\i}rez
\footnote{
Departamento de Matem\'aticas/INDESS (Instituto Universitario para el Desarrollo Social Sostenible), Universidad de C\'adiz,
E-11406 Jerez de la Frontera (C\'{a}diz, Spain).
E-mail: juandios.diaz@uca.es.}\\
J. I. Garc\'{\i}a-Garc\'{\i}a
\footnote{
Departamento de Matem\'aticas/INDESS (Instituto Universitario para el Desarrollo Social Sostenible),
Universidad de C\'adiz, E-11510 Puerto Real  (C\'{a}diz, Spain).
E-mail: ignacio.garcia@uca.es.}\\
A. S\'{a}nchez-R.-Navarro
\footnote{
Departamento de  Ingenier\'{\i}a Inform\'{a}tica/INDESS (Instituto Universitario para el Desarrollo Social Sostenible), Universidad de C\'adiz,
E-11406 Jerez de la Frontera (C\'{a}diz, Spain).
E-mail: alfredo.sanchez@uca.es}
\\
A. Vigneron-Tenorio
\footnote{
Departamento de Matem\'aticas/INDESS (Instituto Universitario para el Desarrollo Social Sostenible), Universidad de C\'adiz,
E-11406 Jerez de la Frontera (C\'{a}diz, Spain).
E-mail: alberto.vigneron@uca.es.}
}
\date{}

\begin{document}
\maketitle
\abstract{A proportionally modular affine semigroup is the set of nonnegative integer solutions of a modular Diophantine inequality $f_1x_1+\cdots +f_nx_n \mod b \le g_1x_1+\cdots +g_nx_n$ where $g_1,\dots,g_n,$ $f_1,\ldots ,f_n\in \Z$ and $b\in\N$. In this work, a geometrical characterization of these semigroups is given. Moreover, some algorithms to check  if a semigroup $S$ in $\N^n$, with $\N^n\setminus S$ a finite set, is a proportionally modular affine semigroup are provided by means of that geometrical approach.}



\section*{Introduction}

An affine semigroup $S$ is a commutative subsemigroup of $\N^n$, that is, $S$ is a subset of $\N^n$ containing the origin, and such that $x+y=y+x\in S$ for all $x,y\in S$. The semigroups satisfying $\N^n\setminus S$ is a finite set are called generalized numerical semigroups or $\N^n$-semigroups (see \cite{GenSemNp} and \cite{Csemigroup}). If $n=1$, they are called numerical semigroups.

Proportionally modular numerical semigroups were introduced by Rosales et al. in \cite{RoGarGarUrb}. These numerical semigroups are the nonnegative integer solutions of Diophantine modular inequality $a x\mod b  \le cx,$ where $a,c\in \Z$ and $b\in \N$. Several research studies have been done about them from multiple points of view. For example, it has been proved their relations with the numerical semigroups generated by intervals, with Bezout's sequences, etc., and some different ways to determine if a numerical semigroup is a proportionally modular numerical semigroup have been given. A comprehensive compilation of these numerical semigroups it is shown in \cite{libro_rosales}.

The natural generalization of proportionally modular numerical semigroups to higher dimension is given in \cite{PropAffinSemig}: a proportionally modular affine semigroup is the set of nonnegative integer solutions of a modular Diophantine inequality $f_1x_1+\cdots +f_nx_n \mod b \le g_1x_1+\cdots +g_nx_n$ where $g_1,\dots,g_n,f_1,\ldots ,f_n\in \Z$ and $b\in\N$. In that paper, the authors determine some algorithms to obtain the minimal generating set of a semigroup from its modular Diophantine inequality. Besides, some properties related with its associated ring are studied.

The main goal of this work is to give algorithmic methods for checking if a semigroup is a proportionally modular affine semigroup. In order to obtain such algorithms, we provide a geometrical characterization of these semigroups. In particular, we prove that an affine semigroup is a proportionally modular semigroup if and only if it is the union of the natural points belonging to some translations of the polyhedron delimited by two hyperplanes (Theorem \ref{CharactPropAffMod}).

Based on Theorem \ref{CharactPropAffMod}, Algorithm \ref{algoritmo_todos_intervalos} and Algorithm \ref{algoritmo_NO_todos_intervalos} check if an $\N^n$-semigroup is a proportionally modular semigroup. These algorithms solve the problem for two different cases. The first one tests if $S$ is a proportionally modular semigroup when $S$ includes no elements in the canonical basis of $\R^n$. For this case, we prove that the proportionally modular semigroups are the semigroups obtained from a special kind of polytopes. In a way, this is equivalent to that happens for proportionally modular numerical semigroups.
The second algorithm can be applied when some element of the canonical basis belongs to $S$.

In any case, both algorithms solve the problem by means of finding real solutions to systems of polynomial inequalities constructed from some technical results (i.e. Corollary \ref{corolario_caso_bandas} and Theorem \ref{casoNnprisma}). These inequalities are the explicit conditions that have to be satisfied by the elements in $S$ and in $\N^n\setminus S$ so that a semigroup $S\subset \N^n$ is a proportionally modular semigroup. In fact, this work presents some algorithms to determine if a finite subset of $\N^n$ fulfils some specific geometrical configurations and arrangements. Several references about solving systems of polynomial inequalities can be found in \cite{Deloera}.

For proportionally modular numerical semigroups, we introduce the concepts of minimal and maximal intervals defining them. These intervals have an important role in the algorithms above mentioned. Furthermore, we provide an algorithm for computing the sets of these minimal and maximal intervals.

The results of this work are illustrated with several examples. To this aim,
we have used the library \texttt{PropModSemig.m} (\cite{prog_prop_mod_afines}) developed by the authors in Mathematica (\cite{mathematica}).

The content of this work is organized as follows. Section \ref{seccion_numericos} provides some basic definitions and results related to proportionally modular numerical semigroups, including algorithmic methods for computing the sets of minimal and maximal intervals defining them. Section \ref{seccion_notacion_Nsemig} gives several definitions and notations related to affine semigroups. Furthermore, most inequalities used in the main algorithms in the following sections are defined in this section. Section \ref{seccion_caracterizacion_geometrica} shows the geometrical characterization of proportionally modular affine semigroups. The algorithms for checking if a semigroup is proportionally modular affine semigroup are introduced in sections \ref{SectionCase1} and \ref{SectionCase2}. Section \ref{seccionTriangulos} studies the two dimensional case for a better understanding of the case 2 solved in Section \ref{SectionCase2}.

\section{Initial results on proportionally modular numerical semigroups}\label{seccion_numericos}

In this section we introduce some results and definitions about numerical semigroups those are useful for the understanding of this work.

Let $\R$, $\Q$ and $\N$ be the sets of real numbers, rational numbers, and nonnegative integers, respectively.
Denote by $\R_\geq$ and $\Q_\geq$ the set of nonnegative elements of $\R$ and $\Q$, and by $\N^*$ the set $\N\setminus\{0\}.$ We denote by $[n]$ the set $\{1,\ldots n\}$ for any $n\in \N$.

A numerical semigroup $S$ is called half-line semigroup if there exists $m\in \N$ such that $S=\N\cap [m,\infty).$ Given an interval $\I\subset \R_\ge$, denote by $\CaS(\I)$ the numerical semigroup $\bigcup _{i\in \N}i\I\cap \N.$ Half-line semigroups can be characterized by a property of the intervals defining them.

\begin{lemma}
$S$ is a non proper half-line semigroup if and only if there exists an interval $[p,q]$ such that $p>1,$ $S=\CaS([p,q])$ and $S=\CaS([p,q'])$ for all $q'\ge q.$
\end{lemma}

\begin{proof}
Assume that $S$ is a half-line semigroup, so there exists an integer $m>1$ such that $S$ is minimally generated by $\{m,\ldots ,2m-1\}$. Then, the interval $[m,2m-1]$ satisfies the lemma.

If $S=\CaS([p,q])$ and $S=\CaS([p,q'])$ for all $q'\ge q,$ the set $\N\setminus S$ is $\{1,\ldots ,m-1\}.$ So, $S$ is a non proper half-line semigroup.
\end{proof}

We say that an interval $[p,q]$ with $p>1$ is a half-line interval if $\CaS([p,q])=\CaS([p,q'])$ for all $q'\ge q.$

In \cite{libro_rosales}, it is proved that proportionally modular numerical semigroups are numerical semigroups generated by a closed interval with lower endpoint greater than 1, that is, for any proportionally modular numerical semigroup $T$ given by the equation $a x\mod b  \le cx,$ there exists an interval $[d,e]$ such that $d>1$ and $T=\CaS([d,e]).$ In this work, we assume that the lower endpoint of every interval defining a numerical semigroup is greater than 1.
Note that if $0<a\le c$ or $a=0$, $T$ is the proper numerical semigroup $\N$, and for $c\le 0,$ $T=\{0\}.$
The relationship between proportionally modular numerical semigroups and numerical semigroups generated by intervals is expressed in the following lemmas.

\begin{lemma}\cite[Lemma 5.9]{libro_rosales}
If $0<c<a<b,$ the
proportionally modular numerical semigroup defined by $ax\mod b\le cx$ is $\CaS([d,e])$ with $d=b/a$ and $e=b/(a-c).$
\end{lemma}

\begin{lemma}\cite[Lemma 5.12]{libro_rosales}\label{IntervalToEquation}
Let $a_1,a_2,b_1,$ and $b_2$ be positive integers such that $b_1/a_1<b_2/a_2$. Then, the semigroup $\CaS([b_1/a_1,b_2/a_2])$ is the proportionally modular numerical semigroup defined by $a_1b_2x \mod (b_1b_2)\le (a_1b_2-a_2b_1)x .$
\end{lemma}

Note that for a closed interval $[d,e],$ the intersection $j[d,e]\cap (j+1)[d,e]$ is empty for every nonnegative integer $j\le \lfloor d/(e-d) \rfloor,$ or, equivalently, $j[d,e]\cap (j+1)[d,e]$ is not empty iff $j\in [ \lceil d/(e-d) \rceil ,\infty )\cap \N.$ We denote by $\phi([d,e])$ the $\min \{j\in \N\mid (j+1)[d,e]\cap j[d,e]\neq \emptyset\},$ that is, $\phi([d,e])= \lceil d/(e-d) \rceil$.

Really, there exists a few amount of proportionally modular numerical semigroups compared to the number of numerical semigroups. Table \ref{tabla_prop_mod_vs_numerical} shows the comparison of both sets up to genus $44$. This table has been computed in a cluster of computers (\cite{super}) by using our modified version of \cite{ProgFroHiv} based on \cite{Fromentin}.
\begin{table}[h]
\centering
{\footnotesize
\begin{tabular}{|c|c|c||c|c|c||c|c|c|}
  \hline
  genus & no. sem. & no. prop. & genus & no. sem. & no. prop. & genus & no. sem. & no. prop. \\
   & & mod. sem. & & & mod. sem.  & & & mod. sem.  \\
  \hline
  0 & 1 & 1 & 15 & 2857 & 87 & 30 &  5646773 & 359 \\ \hdashline
  1 & 1 & 1 & 16 & 4806 & 93 & 31 & 9266788 & 353 \\ \hdashline
  2 & 2 & 2 & 17 & 8045 & 105 & 32 & 15195070 & 375 \\ \hdashline
  3 & 4 & 4 & 18 & 13467 & 125 & 33 & 24896206 & 401 \\ \hdashline
  4 & 7 & 6 & 19 & 22464 & 130 & 34 & 40761087 & 405 \\ \hdashline
  5 & 12 & 9 & 20 & 37396 & 145 & 35 & 66687201 & 445 \\ \hdashline
  6 & 23 & 15 & 21 & 62194 & 169 & 36 & 109032500 & 507 \\ \hdashline
  7 & 39 & 18 & 22 & 103246 & 173 & 37 & 178158289 & 499 \\ \hdashline
  8 & 67 & 22 & 23 & 170963 & 188 & 38 & 290939807 & 527 \\ \hdashline
  9 & 118 & 32 & 24 & 282828 & 224 & 39 & 474851445 & 573 \\ \hdashline
  10 & 204 & 36 & 25 & 467224 & 218 & 40 & 774614284 & 566 \\ \hdashline
  11 & 343 & 42 & 26 & 770832 & 238 & 41 & 1262992840 & 604 \\ \hdashline
  12 & 592 & 57 & 27 & 1270267 & 275 & 42 & 2058356522 & 682\\ \hdashline
  13 & 1001 & 58 & 28 & 2091030 & 273 & 43 & 3353191846 & 655 \\ \hdashline
  14 & 1693 & 69 & 29 & 3437839 & 303 & 44 & 5460401576 & 709 \\ %\hdashline
  \hline
\end{tabular}
}
\caption{Proportionally modular numerical semigroups up to genus $44$.}\label{tabla_prop_mod_vs_numerical}
\end{table}

In order to achieve the main goal of this work, we need to improve  the knowledge of the proportionally modular numerical semigroups. In particular, we have to introduce the minimal and maximal intervals defining them.

Given an open interval $]p,q[\subset \R_\ge$, the numerical semigroup $\CaS(]p,q[)$ is called opened modular numerical semigroup. For a given numerical semigroup, $\EH(S)\subset \N\setminus S$ is the set of elements $\N\setminus S$ such that $S\cup \{x\}$ is a semigroup.

In \cite{RoGarGarUrb}, it is given a characterization of the numerical semigroups defined from closed intervals.

\begin{proposition}{\cite[Proposition 23]{RoGarGarUrb}}\label{sucesiones}
Let $S\neq \N$ be a numerical semigroup minimally generated by $\{n_1,\ldots ,n_t\}$. Then, $S=\CaS([p,q])$ with $q>p>1$ if and only if the following conditions hold:
\begin{enumerate}
\item for all $i\in[t]$, there exists $k_i\in [n_i-1]$ such that $n_i/k_i\in [p,q],$
\item for all $x\in \EH(S)$ and $k_x\in [x-1]$, $x/k_x\notin [p,q].$
\end{enumerate}
\end{proposition}

\begin{remark}
Previous proposition means that for each possible closed interval $[p,q]$ such that $p>1$ and $S=\CaS([p,q])$, it has to exist a sequence $p\le x_1/y_1\leq \cdots \leq x_{l}/y_{l}\le q$ and two integers $h,r\in \N$ such that:
\begin{enumerate}
\item $\{x_h,\ldots ,x_{h+r}\}= \{n_1,\ldots ,n_t\}$,
\item $y_{h+i}\in [x_{h+i}-1]$ for $i\in \{0,\ldots ,r \}$,
\item $x_{h-1}/y_{h-1}\neq x_{h}/y_{h},$
\item $x_{h+r}/y_{h+r}\neq x_{h+r+1}/y_{h+r+1},$
\item and, $x/k_x\notin [p,q]$ for all $x\in \EH(S)$ and $\forall k_x\in [x-1].$
%\item and, $x/k_x\notin [x_h/y_h,x_{h+r}/y_{h+r}]$ for all $x\in \EH(S)$ and $\forall k_x\in [x-1].$
\end{enumerate}

Note that, if $S$ is a proportionally modular numerical semigroup, there exists a finite set of intervals $[x_h/y_h,x_{h+r}/y_{h+r}]$ satisfying these conditions. Following this idea, an algorithm to check if a numerical semigroup is proportionally modular is given in \cite[Algorithm 24]{RoGarGarUrb}.
\end{remark}

Now, we introduce the concepts of minimal and maximal intervals defining proportionally modular numerical semigroups.

\begin{definition}
Given $S$ a proportionally modular numerical semigroup, a closed interval $[\widetilde{p},\widetilde{q}]$ is a minimal (closed) interval defining $S$ if
$[\widetilde{p},\widetilde{q}]$ is a minimal element respect to inclusion in $\{ [p,q]\subset (1,\infty)\mid S=\CaS([p,q]) \}.$
\end{definition}

Note that for any $S$ non proper proportionally modular numerical semigroup, the set of minimal intervals defining $S$ is finite. We denote by $\widetilde{L}_S$ this set.

\begin{lemma}\label{holgura}
Given $p,q\in \R_>$ with $\CaS([p,q])$ a non proper semigroup,
there exists an unique maximal open interval (respect to inclusion) $]\widehat{p},\widehat{q}[$ such that $[p,q]\subset ]\widehat{p},\widehat{q}[$ and $\CaS(]\widehat{p},\widehat{q}[)=\CaS([p,q])$.
\end{lemma}

\begin{proof}
Let $i_0$ be the minimal integer satisfying $iq\le (i+1)p$, $X=\N\setminus \CaS([p,q])$ and $\widehat{p}=p-\min \{\frac{ip-s}{i} \mid s\in X\cap ((i-1)q,ip), \, i=1,\ldots ,i_0 \}$.

Assuming $[p,q]$ is a non half-line interval, the lemma holds for $\widehat{q}=q+\min \{ \frac{s-iq}{i} \mid s\in X\cap (iq,(i+1)p), \, i=1,\ldots ,i_0-1 \}$. In other case, $\widehat{q}=\infty$ has to be consider.
\end{proof}

\begin{remark}\label{def_intervalos_gorro}
Given $S$ a non proper proportionally modular numerical semigroup, $\widehat{L}_S$ denotes the finite set of open intervals $\cup_{[\widetilde{p},\widetilde{q}]\in \widetilde{L}_S} \{]\widehat{p},\widehat{q}[\}$ where $]\widehat{p},\widehat{q}[$ is the unique interval obtained from the proof of Lemma \ref{holgura}. \end{remark}

Algorithm \ref{algoritmo_intervalos_max_m_yin} computes the sets of minimal and maximal intervals defining a proportionally modular numerical semigroup. This algorithm is based on \cite[Algorithm 24]{RoGarGarUrb}. Note that both algorithms are not equal. For example, if you apply the algorithm \cite[Algorithm 24]{RoGarGarUrb} to the proportionally modular numerical semigroup minimally generated by $\{2,3\}$, the obtained closed intervals determining this semigroup are $[3/2,2]$, $[2,3]$ and $[3/2,3]$, but, trivially, the last one is not a minimal interval.

\begin{algorithm}[h]\label{algoritmo_intervalos_max_m_yin}
	\KwIn{$\Lambda_S$ the minimal generating set of a non proper numerical semigroup $S$.}
	\KwOut{If $S$ is proportionally modular, the sets $\widetilde{L}_S$ and $\widehat{L}_S$, the empty set in other case.}

\Begin{
	Compute $\EH(S)$\;
	$A \leftarrow \{(a,k_a)\mid a\in \Lambda_S \cup (\EH(S)\setminus \{1\}),\, k_a\in[a-1] \}$\;
    $A \leftarrow \mbox{Sort }A$ respect to $(a,k)\preceq (a',k')$ if and only if $a/k< a'/k'$, or $a/k= a'/k'$ and $a<a'$\;
	$L \leftarrow  \{\big((x_h,y_h),(x_{h+r},y_{h+r})\big)\mid (x_h,y_h),(x_{h+1},y_{h+1}),\ldots ,(x_{h+r},y_{h+r})\in A,\, \{x_h,\ldots
,x_{h+r}\}= \{n_1,\ldots ,n_t\},\, x_{h-1}/y_{h-1}\neq x_{h}/y_{h},\, x_{h+r}/y_{h+r}\neq x_{h+r+1}/y_{h+r+1}\}$\;
    \If{$L=\emptyset$}
        {
        \Return $\emptyset$
        }
    $\widetilde{L}_S \leftarrow \min_{\subseteq}\{ [a/k, a'/k']\mid \big( (a,k),(a',k') \big)\in L \}$\;
    $\widehat{L}_S \leftarrow \{ ]\widehat{p}, \widehat{q}[ \mid [\widetilde{p},\widetilde{q}]\in \widetilde{L}_S\}$ (Remark \ref{def_intervalos_gorro})\;
	\Return $\widetilde{L}_S$ and $\widehat{L}_S$ \;  % $\{S\in{\cal A}\mid \fn(S)=F\}$\;
}
\caption{Test if a semigroup is a proportionally modular numerical semigroup. In that case, compute the sets of minimal and maximal intervals defining it.}
\end{algorithm}

\begin{example}
Let $S$ be the numerical semigroup minimally generated by $\{10,11,12,13,27\}.$ The set $\EH(S)$ is $\{28,29\}$. So, Algorithm \ref{algoritmo_intervalos_max_m_yin} determines that $S$ is a proportionally modular numerical semigroup defined by the minimal intervals $\widetilde{L}_S= \{[\frac{27}{25},\frac{10}{9}],[10,\frac{27}{2}]\}$ and the corresponding maximal intervals $\widehat{L}_S= \{]\frac{14}{13},\frac{29}{26}[,]\frac{29}{3},14[\}$. In fact, $S$ is the numerical semigroup obtained from the inequality $11 x \mod 110 \leq 3 x$.
\end{example}

\section{Fixing notations for affine semigroups}\label{seccion_notacion_Nsemig}

Let $\{e_1,\ldots ,e_n\}\subset\N^n$ be the canonical basis of $\R^n$. We define $\langle e_{i_0},\ldots ,e_{i_t}\rangle _\R$ as the $\R$-vector space generated by $\{e_{i_0},\ldots ,e_{i_t}\}$.

For a subset $A\subseteq\Q^n$, denote by $\CH(A)$ the convex hull of the set $A$, that is, the smallest convex subset of $\Q^n$ containing $A$, and by $\VSet(A)$ the vertex set of $\CH(A)$.
A polyhedron is a region defined by the intersection of a finitely many closed half-spaces, and a polytope is the convex hull of a finite number of points, or, equivalently, it is a bounded polyhedron (see \cite{Bruns} for details). From these definitions, it is easy to prove that for checking if a finite set of points are in the same region defined by a hyperplane, it is enough to test if the vertices of its convex hull hold that property.

For $\{a_1,\ldots ,a_k\}\subset [n]$ and $x\in \N^n$ we define $\pi_{\{a_1,\ldots ,a_k\}}(x)=(x_{a_1},\ldots,x_{a_k})$, that is, the projection of $x$ on its $\{a_1,\ldots ,a_k\}$ coordinates. Fixed $A\subset \Q^n,$ $\pi_{\{a_1,\ldots ,a_k\}}(A)=\{\pi_{\{a_1,\ldots ,a_k\}}(x)\mid x\in A\}$. Besides, $\sigma_{\{a_1,\ldots ,a_k\}}(A)$ denotes the set $\{\pi_{\{a_1,\ldots ,a_k\}}(x)\mid x\in A\mbox{ and } x_i=0,\, \forall i\in [n]\setminus \{a_1,\ldots ,a_k\} \}.$

Fixed $f(x_1,\dots,x_n)=f_1x_1+\dots+f_nx_n$ and $g(x_1,\dots,x_n)=g_1x_1+\dots+g_nx_n$ with $g_1,\dots,g_n,f_1,\ldots ,f_n\in \Z$, let $S\subset \N^n$ be the semigroup defined by the inequality $f(x)\mod b\leq g(x)$ where $x=(x_1,\dots,x_n)$. We assume that $f_i=f_i\mod b$ for all $i\in [n]$, so, these coefficients are nonnegative integers. It is easy to prove that if $g_i>0$ for all $i\in[n],$ $\N^n\setminus S$ is a finite set. Note that for this semigroup and for every $i\in [n]$, the set
$S\cap \langle e_i \rangle_\R$
is isomorphic to the proportionally modular numerical semigroup given by the set of natural solutions of $f_i x_i\mod b \le g_i x_i$.

We denote by $G_i$ the hyperplane with linear equation $g(x)=ib,$ and by $G_i^+$ the closed half-space defined by $g(x)\ge ib$ ($G_i^-$ is the opened half-space $g(x)< ib$). Analogously, $F_i$ is the hyperplane $f(x)=ib$ (respectively $D_i\equiv f(x)-g(x)=ib$), $F_i^+$ is $f(x)\ge ib$ (respectively $D_i^-\equiv f(x)-g(x)\le ib$), and $F_i^-$ is $f(x)< ib$ (respectively $D_i^+\equiv f(x)-g(x)> ib$).

In \cite{Bruns}, it is proved that given a polytope $\P\subset \R^n_\ge$, the monoid $\bigcup _{i\in \N}i\P\cap \N^n$ is an affine semigroup if and only if $\P\cap \tau\cap \Q^n_\ge\neq \emptyset$ for all $\tau$ extremal ray of the rational cone generated by $\P$. Equivalently to definition of $\CaS(\I)$ for a closed real interval $\I$, $\CaS(\P)$ defines the affine semigroup $\bigcup _{i\in \N}i\P\cap \N^n$.

For a set of closed intervals $L=\{[p_1,q_1], \ldots ,[p_t,q_t]\}$, we denote by $\P_L$ the polytope
$\CH(\cup_{i\in [t]} \{ p_ie_i,q_ie_i\})$, and by $\phi(L)=\max\{\phi([p_i,q_i])\mid i\in [t]\}.$
So, $\CH ( i\P_L\cup (i+1)\P_L) = i\P_L\cup (i+1)\P_L$ for every integer $i\ge \phi(L),$ and $\N^t\setminus \bigcup _{i\in \N}i\P_L\subset \CH ( \{0\}\cup \phi(L)\P_L).$ Moreover, the polytopes $i\P_L$ can be determined by hyperplanes and half-spaces:
\begin{itemize}
\item $H_{1iL}$ is the hyperplane containing the set of points $\{ip_1e_1,\ldots ,ip_te_t\}$, which equation is denoted by $h_{1iL}(x)=0$; $H_{1iL}^+$ is the closed half-space delimited by $H_{1iL}$ not containing the origin,
\item $H_{2iL}$ is the hyperplane containing $\{iq_1e_1,\ldots ,iq_te_t\}$, its equation is $h_{2iL}(x)=0$, and $H_{2iL}^-$ is the closed half-space delimited by $H_{2iL}$ containing the origin.
\end{itemize}
So, the convex set $i\P_L$ is $\R^t_\ge\cap H_{1iL}^+\cap H_{2iL}^-.$ Note that some expressions for $h_{1iL}(x)=0$ and $h_{2iL}(x)=0$ can be easily constructed by using linear algebra. Consider
\begin{equation}\label{eq_hs}
h_{1iL}(x):= p_1\cdots p_t(i-\sum_{j\in[t]} x_j/p_j) \mbox{ and } h_{2iL}(x):= q_1\cdots q_t(i-\sum_{j\in[t]} x_j/q_j).
\end{equation}
Furthermore, $H_{1iL}^+$ is defined by $h_{1iL}(x)\le 0,$ and $H_{2iL}^-$ by $h_{2iL}(x)\ge 0.$
We also consider the opened half-spaces $H_{1iL}^-$ defined by $h_{1iL}(x)> 0,$ and $H_{2iL}^+$ by $h_{2iL}(x)<0.$

Given any $P\in \N^n$, $\kappa_{L}(P)$ denotes the maximal integer $i$ such that $h_{1iL}(P)<0$. In case $h_{11L}(P)\ge 0$, $\kappa_{L}(P)=0$. Besides, we define by $\theta_{L}(P)$ the function such that $\theta_{L}(P)=1$ if there exists $i\in \N$ with $P\in i\P_L$, and $\theta_{L}(P)=0$ in other case. Note that $\theta_{L}(P)$ can be determined in the following way: compute the sets $\tau_P\cap \P_L=\overline{AB}$ and
$\{k\in\N|\frac{||P||}{||B||}\leq k \leq \frac{||P||}{||A||}\}$, where $\tau_P$ is the ray containing $P$, and $||X||$ is the Euclidean norm of $X,$ that is, $||X||^2=||(x_1,\dots,x_t)||^2=\sum_{j=1}^t x_j^2$; if $\{k\in\N|\frac{||P||}{||B||}\leq k \leq \frac{||P||}{||A||}\}=\emptyset$, $\theta_{L}(P)$ is $0$, and $1$ in other case.

Assumed $n>t,$ let
$\mu_{1,i}e_t+\mu_{2,i}e_i$ and $-\nu_{1,i}e_t+\nu_{2,i}e_i$ where $i\in \{t+1,\ldots ,n\}$ and $\mu_{1,i},\mu_{2,i},\nu_{1,i},\nu_{2,i}\in\Q$ be some vectors, and $\tau_{1iL}(x)=0$ and $\tau_{2iL}(x)=0$ be equations of the hyperplanes defined by these vectors and the sets of points $\{iq_1e_1,\ldots ,iq_te_t\}$ and $\{ip_1e_1,\ldots ,ip_te_t\}$ respectively. For every $i\in \Z$, consider
\begin{equation}\label{eqtaus}
\begin{array}{c}
\tau_{1iL}(x):=q_1\cdots q_t\mu_{2,t+1}\cdots \mu_{2,n}\displaystyle{\Big(i-\sum_{j\in[t]} \frac{x_j}{q_j}+ \sum_{j=t+1}^n \frac{\mu_{1,j}}{q_t  \mu_{2,j}} x_j \Big)},
\\
\\
\tau_{2iL}(x):=p_1\cdots p_t\nu_{2,t+1}\cdots \nu_{2,n}\displaystyle{\Big(i-\sum_{j\in[t]} \frac{x_j}{p_j}- \sum_{j=t+1}^n \frac{\nu_{1,j}}{p_t  \nu_{2,j}} x_j \Big)}.
\end{array}
\end{equation}

\section{A geometrical characterization of proportionally modular affine semigroups}\label{seccion_caracterizacion_geometrica}

Let $S$ be the proportionally modular semigroup given by the modular inequality $f(x)\mod b\leq g(x)$ where $x=(x_1,\dots,x_n)$, $g_1,\dots,g_n,f_1,\ldots ,f_n\in \Z$ and $b\in\N.$ Assume that $f_i=f_i\mod b$ for all $i\in [n]$, so, these coefficients are nonnegative integers.

As in previous sections, we denote by $F_i$ the hyperplane with linear equation $f(x)=ib$ (respectively $D_i\equiv f(x)-g(x)=ib$), and by $F_i^+$ the closed half-space defined by $f(x)\ge ib$ (respectively $D_i^-\equiv f(x)-g(x)\le ib$). Fixed these hyperplanes, $\P_i$ denotes the polyhedron $F_i^+\cap D_i^-$ with $i\in \Z.$
Note that, since $b>0,$ and $f(e_j)\ge 0$ for all $j\in[n],$ set a negative integer $i$, every $P\in \P_i\cap \N^n$ satisfies $P\in F_0^+\cap D_0^-$. So, $\cup _{i\in \N} (F_i^+\cap D_i^-)\cap \N^n=\cup _{i\in \Z} (F_i^+\cap D_i^-)\cap \N^n$. Besides, for all $i\in \N,$ the points $P$ in $F_i^+\cap D_i^-$ hold $g(P)\ge 0.$

\begin{lemma}\label{InPropAffMod}
$P\in S$ if and only if there exists $i\in \N$ such that $P\in \P_i\cap \N^n.$
\end{lemma}

\begin{proof}

For any $P\in\N^n,$ there exist two nonnegative integers $i$ and $r$ such that $f(P)=ib+r$ with $r\in[0,b).$ Assume $P\in S,$ so $f(P)\mod b\le g(P).$ Furthermore, $0\le f(P)\mod b=r=f(P)-ib\le g(P)$, and then $P$ belongs to $\P_i\cap \N^n.$

Consider now any $P\in \P_i\cap \N^n$ with $i\in \N$, so $0\le f(P)-ib \le g(P).$
Since $f(P)\mod b=f(P)-ib\mod b\le b$, if $g(P)\ge b$, trivially, $P\in S.$ Suppose that $g(P)<b.$ In that case, $0\le f(P)-ib \le g(P)<b$, and again $f(P)\mod b=f(P)-ib\mod b\le g(P)$, and $P\in S.$
\end{proof}

Now, we have the necessary tools to introduce a geometrical characterization of proportionally modular semigroups in the next result.

\begin{theorem}\label{CharactPropAffMod}
$S\subset \N^n$ is a proportionally modular affine semigroup if and only if there exist two linear functions with integer coefficients $f(x)$ and $d(x)$ with $f_1,\ldots, f_n\ge 0$, an integer $b>0$ and two families of half-spaces $\{F_i^+\}_{i\in \N}$ and $\{D_i^-\}_{i\in \N},$ where $F_i^+\equiv f(x)\ge ib$ and $D_i^-\equiv d(x)\le ib$, such that $S= \cup _{i\in \N} (F_i^+\cap D_i^-)\cap \N^n.$
\end{theorem}

\begin{proof}
Given $S$ a proportionally modular semigroup, $S$ is the set of nonnegative integer solutions of an inequality $f(x)\mod b\leq g(x)$ where $f(x)$ and $g(x)$ are linear functions with integers coefficients, and $b\in\N^*$. From Lemma \ref{InPropAffMod}, taking the half-spaces $F_i^+\equiv f(x)\ge ib$ and $D_i^-\equiv (f-g)(x)\le ib$ the result holds.

Conversely, if $S= \cup _{i\in \N} (F_i^+\cap D_i^-)\cap \N^n$ with $F_i^+\equiv f(x)\ge ib$ and $D_i^-\equiv d(x)\le ib$ where $f(x)$ and $d(x)$ are linear functions and $b\in\N^*$, again, by Lemma \ref{InPropAffMod}, $S$ is the proportionally modular semigroup given by the inequality $f(x)\mod b\leq (f-d)(x)$.
\end{proof}

From previous results, the set of gaps of a proportionally modular semigroup can be described geometrically too.

\begin{corollary}\label{GAPS}
Let $S\subset \N^n$ be a proportionally modular affine semigroup, then $(\N^n\setminus S)\cap G_0^+=\cup _{i\in \N^*} (F_i^-\cap D_{i-1}^+\cap G_0^+)\cap \N^n.$
\end{corollary}

\begin{proof}
Note that for any $P\in F_i^-\cap D_{i-1}^+\cap G_0^+,$ $0 \le g(P)<b.$ Moreover, $(F_i^-\cap D_{i-1}^+\cap G_0^+)\cap (F_j^-\cap D_{j-1}^+\cap G_0^+)=\emptyset$ for every nonnegative integers $i\neq j$. In other case, if we suppose $i<j$ and $P\in (F_i^-\cap D_{i-1}^+\cap G_0^+)\cap (F_j^-\cap D_{j-1}^+\cap G_0^+),$ the point $P$ satisfies $ib>f(P),$ $f(P)-g(P)>(j-1)b,$ and $g(P)\ge 0.$ Then, $0\ge (i-j+1)b>g(P)\ge 0,$ but it is not possible.

Assume $P\in F_i^-\cap D_{i-1}^+\cap G_0^+$ for some $i\in \N^*,$  and there exists a nonnegative integer $j\le i-1$ such that $P\in F_j^+.$ Since $f(P)-g(P)>(i-1)b$, if $jb\ge f(P)-g(P),$ $j>i-1$. So, $P\notin D_j^-$. % and $P\in (\N^n\setminus S)\cap G_0^+.$
Analogously, if $P\in F_i^+\cap D_{i}^-$ with $g(P)<b,$ and we suppose there exists an integer $j\ge i$ such that $P\in F_j^-,$ $P$ does not belong to $D_{j-1}^+$ ($P\in D_{j-1}^+$ implies $ib\ge f(P)-g(P)>(j-1)b$, that is, $i>j-1$). So, $\cup _{i\in \N^*} (F_i^-\cap D_{i-1}^+\cap G_0^+)\bigcap \cup _{i\in \N} (F_i^+\cap D_{i}^-\cap G_b^-)= \emptyset$ and $\N^n\cap G_0^+\cap G_b^-= \cup _{i\in \N^*} (F_i^-\cap D_{i-1}^+\cap G_0^+)\bigcup \cup _{i\in \N} (F_i^+\cap D_{i}^-\cap G_b^-).$

Since $\N^n\cap G_0^+\cap G_b^-= \cup _{i\in \N^*} (F_i^-\cap D_{i-1}^+\cap G_0^+)\bigcup \cup _{i\in \N} (F_i^+\cap D_{i}^-\cap G_b^-),$ by Theorem \ref{CharactPropAffMod}, the corollary holds.
\end{proof}

\section{Testing $\N^n$-semigroups for being proportionally modular affine semigroups. Case 1.}\label{SectionCase1}

For this section, we assume that the semigroup $S\subset \N^n$ satisfies that for every $i\in [n]$, $e_i\notin S$ and $S_i$ is a proportionally modular numerical semigroup.

\begin{proposition}\label{CasoBanda}
$S$ is a proportionally modular $\N^n$-semigroup with $b>f_i>g_i>0$ for all $i\in[n]$ if and only if there exists a set $L=\{[p_1,q_1],\ldots, [p_n,q_n]\}$ where $S_i=\CaS([p_i,q_i])$ for all $i\in [n],$ and $S=\bigcup _{i\in \N}i\P_L\cap \N^n.$
\end{proposition}

\begin{proof}
Assume that $S$ is a proportionally modular $\N^n$-semigroup with $f_i>g_i>0$ for all $i\in [n].$ Consider $\P_L$ the polytope $\CH (\cup_{i\in [n]} \{ p_ie_i,q_ie_i\})\subset \R^n$ where $q_i=b/(f_i-g_i)>p_i=b/f_i>1$, then $S_i=\CaS([p_i,q_i])$. Note that $\P_L$ is the set $(F_1^+\cap D_1^-)\cap \R_\ge^n$ with $F_i^+\equiv f(x)\ge ib,$ and $D_i^-\equiv (f-g)x\le ib$, and $i\P_L$ is equal to $(F_i^+\cap D_i^-)\cap \R_\ge^n.$ So, proposition holds by Theorem \ref{CharactPropAffMod}.

Now, we assume that $\CaS([p_1,q_1]),\ldots, \CaS([p_n,q_n])$ are non proper proportionally modular numerical semigroups.
%Let $\P$ be the convex hull of $\cup_{i\in [n]} \{ p_ie_i,q_ie_i\}\subset \R^n_\ge$, and $S$ be the semigroup $\bigcup _{i\in \N}i\P_L\cap \N^n$.
By Lemma \ref{IntervalToEquation}, for each $i\in [n]$, there exist three nonnegative integers $b_i$, $a_i$ and $c_i$ with $b_i>a_i>c_i>0,$ such that $\CaS([p_i,q_i])$ is the set of nonnegative integer of the inequality $a_ix\mod b_i\le c_ix.$ Get $b=\lcm(\{b_1,\ldots ,b_n\})$ and consider the inequalities $(\frac{b}{b_i}a_i)x_i\mod b\le (\frac{b}{b_i}c_i)x_i.$ Let $S$ be the proportionally modular $\N^n$-semigroup defined by the inequality $\sum_{i=1}^n (\frac{b}{b_i}a_i)x_i \mod b \le \sum_{i=1}^n (\frac{b}{b_i}c_i)x_i.$
It is an easy exercise for the lector to prove that $\bigcup _{i\in \N}i\P_L\cap \N^n$ is the proportionally modular $\N^n$-semigroup defined by the inequality $\sum_{i=1}^n (\frac{b}{b_i}a_i)x_i \mod b \le \sum_{i=1}^n (\frac{b}{b_i}c_i)x_i$, and that $b>\frac{b}{b_i}a_i>\frac{b}{b_i}c_i$ for all $i\in[n]$.
\end{proof}

From this proposition, we obtain a procedure to check if an $\N^n$-semigroup $S$ is a proportionally modular semigroup. For that happens, the first necessary condition that $S$ must satisfy is
$S_i$ has to be a non proper proportionally modular numerical semigroup for all $i\in [n].$ If this initial condition is satisfied, we have to determinate if there exist $n$ intervals $L=\{[p_1,q_1],\ldots, [p_n,q_n]\}$ with $q_i>p_i>1$ such that $S_i=\CaS([p_i,q_i])$ and $S=\bigcup _{i\in \N}i\P_L\cap \N^n$. Let $\Lambda_S$ be the minimal generating set of $S$.

\begin{lemma}\label{lemma:bandas}
Let $S$ be an $\N^n$-semigroup with $e_i\notin S$ for all $i\in [n]$. Then, $S$ is a proportionally modular semigroup if and only if there exist $([\widetilde{p}_1,\widetilde{q}_1],\ldots, [\widetilde{p}_n,\widetilde{q}_n])\in \widetilde{L}_{S_1}\times \cdots \times \widetilde{L}_{S_n}$, and a set
$L=\{[p_1,q_1],\ldots, [p_n,q_n]\}$ such that $S_i=\CaS([p_i,q_i])$
for all $i\in [n]$ satisfying:

\begin{enumerate}\label{ecuacionesbandas}
\item $[\widetilde{p}_i,\widetilde{q}_i]\subseteq [p_i,q_i] \subset ]\widehat{p_i},\widehat{q_i}[$ for all $i\in [n]$;\label{ecuacionesbandas1}
\item $\N^n\setminus S\subset \CH ( \{0\}\cup \phi(L)\P_{L} )$;\label{inclusionhuecos}
\item for every $x\in \N^n\setminus S$ and for every $i\in [\phi(L)]$, $x\notin i\P_L$;\label{ecuacionesbandas2}
\item for every  $s\in \Lambda_S$ such that $\theta_{\widetilde{L}}(s)=0$, $\kappa_{\widehat{L}}(s)\neq 0$, and $s/m\in \P_L$ for some $m\in [\kappa_{\widehat{L}}(s)]$.\label{ecuacionesbandas3}
\end{enumerate}
\end{lemma}

\begin{proof}
If $S$ is a proportionally modular $\N^n$-semigroup with $e_i\notin S$ for all $i\in [n]$, by Proposition \ref{CasoBanda}, $S=\cup_{i\in \N}(i\P_L\cap \N^n)$ where $L=\{[p_1,q_1],\ldots, [p_n,q_n]\}$ is such that $S_i=\CaS([p_i,q_i])$. So, Lemma holds.

Assume that $\P_L$ is a polytope satisfying the hypotheses of the lemma and let $S'$ be the $\N^n$-semigroup $S'=\bigcup _{i\in \N}i\P_L\cap \N^n$. Trivially, $S'\cap \langle e_i \rangle_\R$ is isomorphic to $\CaS([p_i,q_i])$ for all $i\in [n].$ By the conditions \ref{inclusionhuecos} and \ref{ecuacionesbandas2}, $\N^n\setminus S\subset\N^n\setminus S'$ and then $S'\subset S$. By the conditions \ref{inclusionhuecos} and \ref{ecuacionesbandas3}, we have that $S\subset S'$ (note that if $\theta_{\widetilde{L}}(s)=1$ for some $s\in\Lambda_S$, there exists $i\in \N$ with $s\in i\P_{\widetilde{L}}\subset i\P_{{L}}$). So, $s\in \bigcup _{i\in \N}i\P_L\cap \N^n$.
\end{proof}

In order to present an algorithm to check if an $\N^n$-semigroup is a proportionally modular semigroup,
for a set of closed intervals ${L}=\{[p_1,q_1],\ldots, [p_n,q_n]\}$ with $S_i=\CaS([p_i,q_i])$ for every $i\in [n]$,
we set a disjoint partition of the region $\mathbb{T}_{{L}}:=\CH ( \{0\}\cup \phi({L})\P_{{L}} )\cap \N^n$:
\begin{itemize}
\item $\mathbb{H}_{1{L}}=\{x\in\N^n\mid x\in H^-_{11{L}}\}$, and $\mathbb{H}_{i{L}}=H^-_{1i{L}} \cap H^+_{2(i-1){L}}\cap \N^n$ for any $i\in \{2,\ldots, \phi({L})\}$,
\item $\mathbb{S}_{i{L}}= H^+_{1i{L}} \cap H^-_{2i{L}}\cap \N^n$ for any $i\in [\phi({L})].$
\end{itemize}
Note that $\mathbb{T}_{{L}}= \{0\}\cup \bigcup _{i\in [\phi({L})]} (\mathbb{H}_{i{L}}\sqcup \mathbb{S}_{i{L}})$ and $\mathbb{S}_{i{L}}=i\P_{{L}}\cap \N^n.$

For any $i$ belonging to $[\phi({L})]$, denote by
$\Ba_{i{L}}$ the set $(\N^n\setminus S)\cap \mathbb{H}_{i{L}}$.
So, a necessary condition for $S$ to be a proportionally modular $\N^n$-semigroup for some $L=\{[p_1,q_1],\ldots, [p_n,q_n]\}$ is $\N^n\setminus S\subset \cup _{i\in \phi({L})} \Ba_{i{L}} \subset \cup _{i\in \phi(\widetilde{L})} \Ba_{i\widetilde{L}}$. In other case, $S\neq\bigcup _{i\in \N}i\P_L\cap \N^n$.

\begin{corollary}\label{corolario_caso_bandas}
Let $S$ be an $\N^n$-semigroup with $e_i\notin S$ for all $i\in [n]$. Then, $S$ is a proportionally modular semigroup if and only if for some $\widetilde{L}=\{[\widetilde{p}_1,\widetilde{q}_1],\ldots, [\widetilde{p}_n,\widetilde{q}_n]\}$ with $([\widetilde{p}_1,\widetilde{q}_1],\ldots, [\widetilde{p}_n,\widetilde{q}_n])\in \widetilde{L}_{S_1}\times \cdots \times \widetilde{L}_{S_n}$, $\N^n\setminus S\subset \cup _{i\in [\phi(\widetilde{L})]} \mathbb{H}_{i\widetilde{L}}$ and
there exists $L=\{[p_1,q_1],\ldots, [p_n,q_n]\}$ with $p_1,\ldots , p_n,$ $q_1, \ldots ,q_n \in \Q$
satisfying the following inequalities:
\begin{enumerate}
\item for all $i\in [n]$ such that $[\widetilde{p}_i,\widetilde{q}_i]$ is a half-line interval, $\widehat{p}_i<p_i\le \widetilde{p}_i<\widetilde{q}_i\le q_i$, in other case, $\widehat{p}_i<p_i\le \widetilde{p}_i<\widetilde{q}_i\le q_i< \widehat{q}_i$;
\item for all $x\in \VSet((\N^n\setminus S)\cap \mathbb{H}_{1\widetilde{L}})$, $h_{11L}(x)>0$, and for all $i\in \{2,\ldots ,\phi(\widetilde{L})\}$ and $x\in \VSet((\N^n\setminus S)\cap \mathbb{H}_{i\widetilde{L}})$, $h_{1iL}(x) >0$ and $h_{2(i-1)L}(x)<0$;

\item \label{condi4b}  for every $s\in \Lambda_S$ such that $\theta_{\widetilde{L}}(s)=0$,    $\kappa_{\widehat{L}}(s)\neq 0$ and $h_{1iL}(s/m) \le 0$ and $h_{2iL}(s/m)\ge 0$ for some $m\in [\kappa_{\widehat{L}}(s)]$.
\end{enumerate}
\end{corollary}

\begin{proof}
The condition $\N^n\setminus S\subset \cup _{i\in [\phi(\widetilde{L})]} \mathbb{H}_{i\widetilde{L}}$ is equivalent to the condition \ref{inclusionhuecos} in Lemma \ref{lemma:bandas}. Furthermore, for every integer $i,$ the sets of inequalities appearing in the second condition are fulfilled by the rational points belonging to $\mathbb{H}_{i{L}}$, while any point that satisfies the inequalities of the third condition belongs to $i\P_L$ for some integer $i$. Then, second and third conditions of the corollary are equivalent to the conditions \ref{ecuacionesbandas2} and \ref{ecuacionesbandas3} of Lemma \ref{lemma:bandas} respectively.
\end{proof}

Algorithm \ref{algoritmo_todos_intervalos} presents a method for checking the conditions of the previous corollary. Note that some steps in this algorithm can be computed in a parallel way. Given a minimal interval $[\widetilde{p},\widetilde{q}]$, we denote by $r_{[\widetilde{p},\widetilde{q}]}$ the inequalities $\widehat{p}<p\le \widetilde{p}<\widetilde{q}\le q$ if $[p,q]$ is a half-line interval, and $\widehat{p}<p\le \widetilde{p}<\widetilde{q}\le q< \widehat{q}$ in other case.

\begin{algorithm}[h]\label{algoritmo_todos_intervalos}
	\KwIn{The minimal generating set $\Lambda_S$ and the set of gaps of $S$ an $\N^n$-semigroup.}
	\KwOut{If $S$ is proportionally modular with $e_i\notin S$ and such that $S_i$ is a proportionally modular numerical semigroup for all $i\in [n]$, a polytope $\P$ such that $S=\bigcup _{i\in \N}i\P\cap \N^n$, the empty set in other case.}

\Begin{
	\If{$e_i\in S$ or $S_i$ is not a proportionally modular numerical semigroup for some $i\in [n]$}
        {\Return $\emptyset$}

    $L\leftarrow \{[p_1,q_1],\ldots, [p_n,q_n]\}$ set of variables\;
    $\Delta\leftarrow\widetilde{L}_{S_1}\times \cdots \times \widetilde{L}_{S_n}$ (Algorithm \ref{algoritmo_intervalos_max_m_yin})\;

    \While{$\Delta\neq \emptyset$}
        {
        $\widetilde{L}\leftarrow \mbox{First}(\Delta)$\;
        $\Lambda=\{s_1,\ldots ,s_k\}\leftarrow \{s\in \Lambda_S \mid \theta_{\widetilde{L}}(s)=0\}$\;
        \ForAll{$i\in [k]$}
            {Compute $\kappa_{\widehat{L}}(s_i)$\;            }
        \If{$\prod _{j\in [k]}\kappa_{\widehat{L}}(s_j)\neq 0$ and $\N^n\setminus S\subset \cup _{i\in [\phi(\widetilde{L})]} \mathbb{H}_{i\widetilde{L}}$}
            {

            $E\leftarrow \{h_{11L}(x)>0\mid x\in \VSet((\N^n\setminus S)\cap \mathbb{H}_{1\widetilde{L}})\}$\;
            $E\leftarrow E\cup \{\{h_{1iL}(x) >0, h_{2(i-1)L}(x)<0\}\mid i\in \{2,\ldots ,\phi(\widetilde{L})\}\mbox{ and } x\in \VSet((\N^n\setminus S)\cap \mathbb{H}_{i\widetilde{L}})\}$\;

            $\Omega\leftarrow [\kappa_{\widehat{L}}(s_1)]\times \cdots \times [\kappa_{\widehat{L}}(s_k)]$\;

            \While{$\Omega\neq\emptyset$}
                {
                $(m_1,\ldots ,m_k)\leftarrow \mbox{First}(\Omega)$\;

                $F\leftarrow  \{\{h_{11L}(s_i/m_i) \le 0, h_{21L}(s_i/m_i)\ge 0\}\mid i\in [k]\} $\;

                $T \leftarrow \mbox{ Solve } \big(\cup_{i\in [n]}\{r_{[\widetilde{p}_i,\widetilde{q}_i]}\}\big)
                \bigcup E \bigcup F$ for $\{p_1,\ldots , p_n,q_1, \ldots ,q_n\}$
                \;
                \If{$(p_1,\ldots , p_n,q_1, \ldots ,q_n)\in T\cap \R^{2n}$}
                    {\Return $\P= \CH \big(\cup_{i\in [n]} \{ p_ie_i,q_ie_i\}\big)$}
                $\Omega\leftarrow \Omega\setminus \{(m_1,\ldots ,m_k)\}$\;
                }
            }
        $\Delta\leftarrow \Delta\setminus \{\widetilde{L}\}$\;
        }
	\Return $\emptyset$\;
}
\caption{Checking if an $\N^n$-semigroup $S$ with $e_i\notin S$, $\forall i \in [n],$ is a proportionally modular semigroup.}
\end{algorithm}

\begin{example}\label{ejemplo_bandas2}
Consider the set of circles in Figure \ref{ejemplo_bandas}, and let $S$ be the $\N^2$-semigroup such that $\N^2\setminus S$ is this set, and its minimal generating set is
\[
\begin{multlined}
\{(0, 8), (0, 9), (0, 10), (0, 11), (0, 12), (0, 15), (1,
7), (1, 8), (1, 9), (1, 10), (1, 11), (1, 14), (2, 6),\\
(2, 7), (2, 8), (2, 9), (2, 10), (3, 6), (3, 7), (3,
8), (3, 9), (4, 5), (4, 6), (4, 7), (4, 8), (5, 4), (5,
5),\\
 (5, 6), (5, 7), (5, 11), (6, 3), (6, 4), (6, 5),
(6, 6), (7, 3), (7, 4), (7, 5), (7, 6), (8, 2), (8, 3),
(8, 4),\\
 (8, 5), (9, 1), (9, 2), (9, 3), (9, 4), (10,
0), (10, 1), (10, 2), (10, 3), (11, 0), (11, 1), (11,
2), (12, 0),\\
 (12, 1), (13, 0), (23, 4), (24, 3), (25,
2), (26, 1), (27, 0) \}.
\end{multlined}
\]%
So, $S_1$ is minimally generated by $\{10,11,12,13,27\}$, and $S_2$ is minimally generated by $\{8,9,10,11,12,15 \}$. Using Algorithm \ref{algoritmo_intervalos_max_m_yin},
we obtain that both $S_1$ as $S_2$ are proportionally modular numerical semigroups with   $\widetilde{L}_{S_1}= \{[\frac{27}{25},\frac{10}{9}],[10,\frac{27}{2}]\}$, $\widehat{L}_{S_1}= \{]\frac{14}{13},\frac{29}{26}[,]\frac{29}{3},14[\}$,  $\widetilde{L}_{S_2}= \{[\frac{12}{11},\frac{15}{13}],[\frac{15}{2},12]\}$, and $\widehat{L}_{S_2}= \{]\frac{13}{12},\frac{7}{6}[,]7,13[\}$, respectively. Algorithm \ref{algoritmo_todos_intervalos} determines that $S$ is a proportionally modular $\N^2$-semigroup when $\widetilde{L}= \left([10,\frac{27}{2}], [\frac{15}{2},12] \right)$ by computing $p_1,q_1,p_2,q_2\in \Q$ such that $\frac{29}{3}<p_1\leq 10 < \frac{27}{2} \leq q_1< 14$, $7<p_2\leq \frac{15}{2}<12 \leq q_2< 13$, and satisfying the other inequalities in Corollary \ref{corolario_caso_bandas}. Moreover, $S$ is given by the inequality $11x+ 15y \mod  110\leq  3x+ 6y$. In Figure \ref{ejemplo_bandas}, the blue line is $g(x)=b$, the green line is $f(x)=kb$ and the red one is $f(x)-g(x)=(k-1)b$ for $k\in \N$.
\begin{figure}[h]
  \centering
\includegraphics[scale=.28]{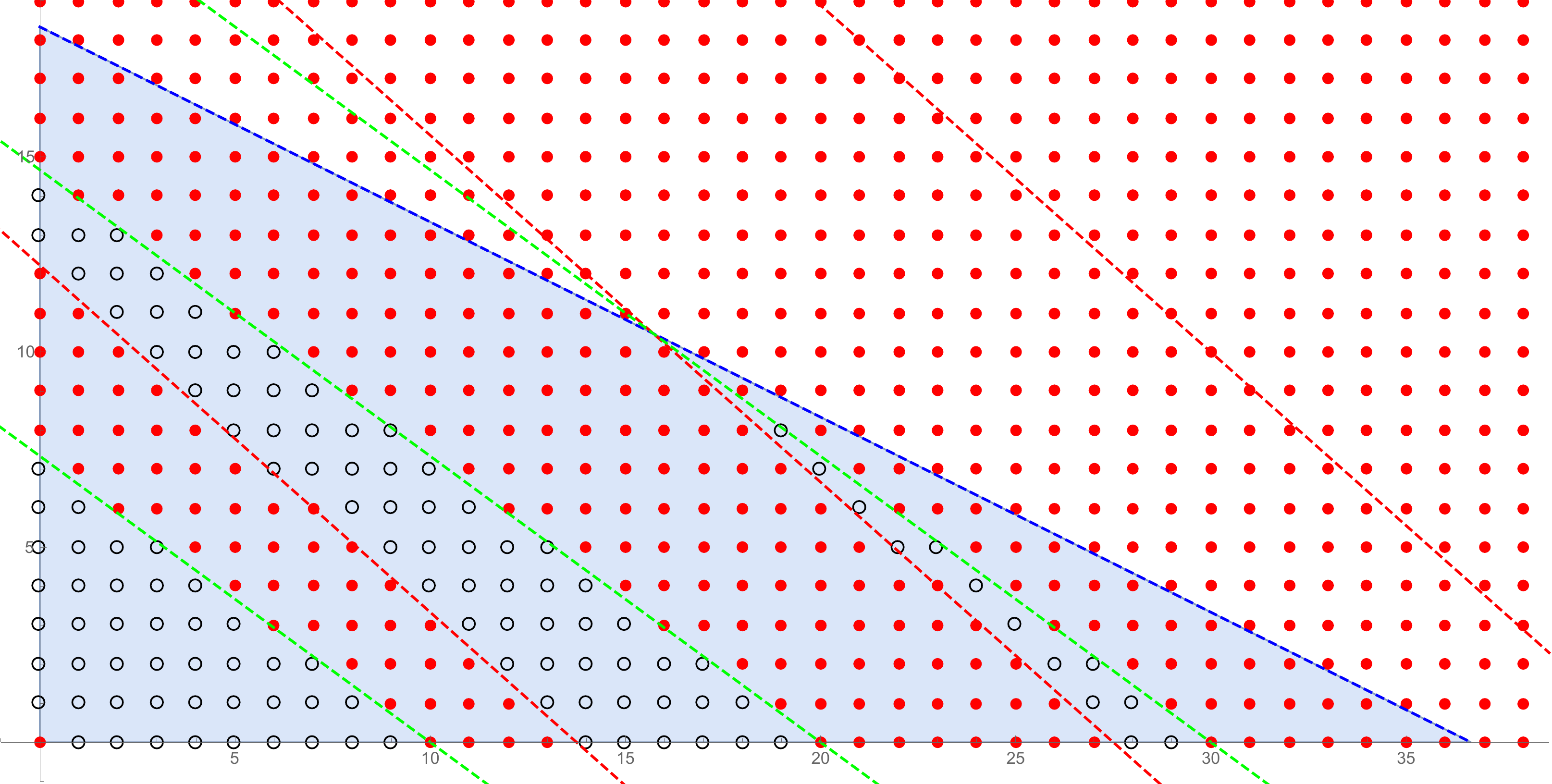}
  \caption{$\N^2$-semigroup given by $11x+ 15y\mod  110\leq  3x+6y$.}\label{ejemplo_bandas}
\end{figure}
Using our software \cite{prog_prop_mod_afines}, we can obtain the above results,
\begin{verbatim}
In[1]:= mgs = {{0, 3}, {0, 4}, {1, 1}, {2, 1}, {4, 0}, {5, 0}, {5, 2},
        {6, 0}, {7, 0}};
In[2]:= gaps = {{0, 1, 0}, {0, 2, 0}, {0, 2, 1}, {0, 5, 0}, {1, 0, 0}, {1, 2, 0},
        {1, 3, 0}, {1, 6, 0}, {2, 0, 0}, {2, 0, 1}, {2, 3, 0}, {3, 0, 0},
        {3, 1, 0}, {3, 4, 0}, {4, 1, 0}};

In[3]:= IsNnProportionallyModularSemigroup[mgs, gaps]

Out[3]= {{163835/16384, 0}, {28367/2048, 0}, {0, 931/128}, {0, 1553/128}}
\end{verbatim}
Note that $p_1= 163835/16384$, $q_1= 28367/2048$, $p_2= 931/128$ and $q_2= 1553/128$.

\end{example}

\section{Some properties of proportionally modular $\N^2$-semigroups}\label{seccionTriangulos}

In order to give an algorithm to check if an $\N^n$-semigroup is a proportionally modular semigroup when some $e_i$ belongs to it, we study the two dimensional case in depth.
Let $S\subset \N^2$ be the non proper proportionally modular semigroup given by $f_1x_1+f_2x_2\mod b \le g_1x_1+g_2x_2$. Again, we assume $f_1=f_1\mod b$, $f_2=f_2\mod b$ and $g_1,g_2>0.$

Without loss of generality, we assume, for example, that $f_1>g_1$ but $g_2\ge f_2$, then $e_2$ belongs to $S$, and for all $x\in \N^2\setminus S,$ $x-e_2\notin S$. Note that if one particularize Theorem \ref{CharactPropAffMod} and Corollary \ref{GAPS} to the two dimensional case, the sets $F_i^-\cap D_{i-1}^+\cap G_0^+$ are triangles with parallel edges, and their bases are the intervals given by $F_i^-\cap D_{i-1}^+\cap \langle e_1\rangle _\R$.

For any integer $k\in (0,f_1/g_1)$, we denote by
$A_{k}$ the point $\frac{b}{f_1g_2-f_2g_1}(kg_2-f_2,f_1-kg_1)=\{g_1x_1+g_2x_2=b\}\cap \{f_1x_1+f_2x_2=kb\}\in\Q^2_\ge$.
%Let $D_{12k}$ be the line with equation $(f_1-g_1)x_1+(f_2-g_2)x_2= (k-1)b.$
The triangle $\T_{k}$ is the convex hull of the vertex set $\{(kb/f_1,0), ((k-1)b/(f_1-g_1),0),A_{k}\},$ and $\ddot{\T}_{k}$ is $\T_{k}\setminus \{\overline{((k-1)b/(f_1-g_1),0) A_{k}}, \overline{(kb/f_1,0) A_{k}}\}$.
This situation is illustrated in Figure \ref{figure_triangles}; the blue line is $g(x)=b$, the green line is $f(x)=kb$ and the red one is $f(x)-g(x)=(k-1)b$ for $k\in \N$.
\begin{figure}[h]
  \centering
\includegraphics[scale=.47]{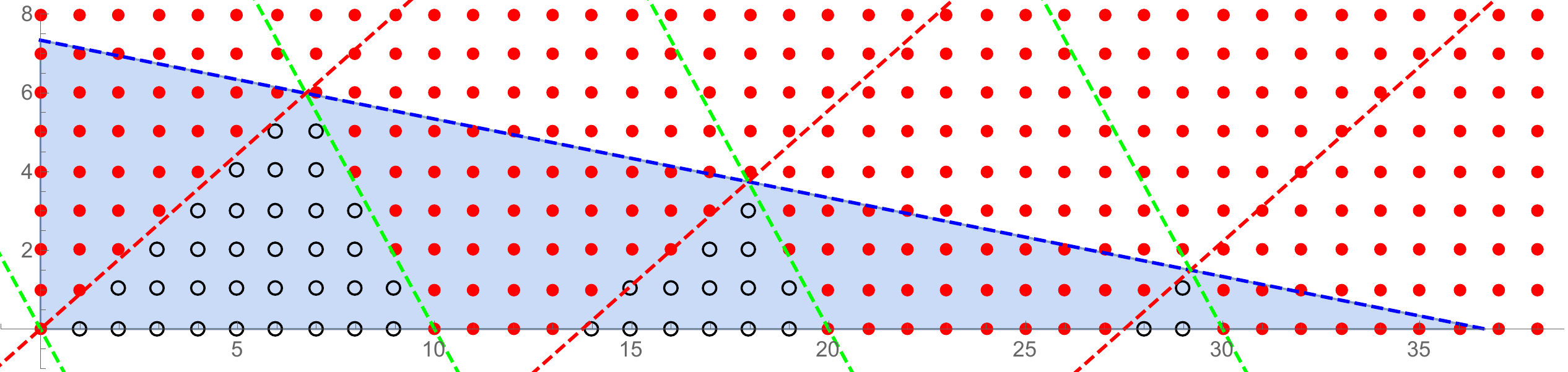}
  \caption{$\N^2$-semigroup given by $11x+ 6y\mod  110\leq  3x+15y$.}\label{figure_triangles}
\end{figure}

So, a proportionally modular semigroup $S$ with $f_1>g_1$ and $g_2\ge f_2$ can be characterized by a finite set of triangles satisfying some conditions.

\begin{lemma}\label{lemma_triangles}
Let $S\neq \N^2$ be a proportionally modular $\N^2$-semigroup $f_1x_1+f_2x_2\mod b\leq g_1x_1+g_2x_2,$ such that $1\in S_2$. For every $\alpha\in \N^2$,
$\alpha\in S$ if and only if $\alpha \notin \cup _{k\in [\lfloor f_1/g_1 \rfloor]} \ddot{\T}_{k}.$
\end{lemma}

\begin{proof}
By Corollary \ref{GAPS}, if $S$ is a proportionally modular semigroup, $(\N^2\setminus S)\cap G_0^+=\cup _{i\in \N\setminus\{0\}} (F_i^-\cap D_{i-1}^+\cap G_0^+)\cap \N^n,$ and $F_i^-\cap G_0^+=\overline{\left(\frac{kb}{f_1},0\right) A_{k}}$ and $D_{i-1}^+\cap G_0^+=\overline{\left(\frac{(k-1)b}{f_1-g_1},0\right) A_{k}}$. These are just the edges of the triangle $\T_{k}$.
\end{proof}

Note that the triangles $\T_i$ can be also determined by the points $(p,0),$ $(q,0)$ and two vectors $(\mu_1,\mu_2)$ and $(-\nu_1,\nu_2)$, with $\mu_1,\nu_1\in [0,1),$ $\mu_2,\nu_2\in (0,1]$, and $\mu_1+\mu_2=\nu_1+\nu_2=1.$

\begin{remark}\label{FromTriangleToEquation}
Given a triangle with vertex set $\T=\{(0,0),(p,0),(\gamma_1,\gamma_2)\}\subset \Q^2_\ge $ such that $0\le \gamma_1\le p,$ and one other point $(q,0)\in \Q^2_\ge$ with $q>p$, a proportionally modular $\N^2$-semigroup can be constructed by using the following method:
the line containing $\{(p,0),(\gamma_1,\gamma_2)\}$ is defined by the equation $\gamma_2x_1+(p-\gamma_1)x_2=p\gamma_2$, and
consider the line $(q-p)\gamma_2x_1+\left( pq-\gamma_1(q-p) \right)x_2=pq\gamma_2$.
Let $r_1,r_2$ be the minimum nonnegative integers such that $\{r_1\gamma_2,r_1(p-\gamma_1),r_1p\gamma_2, r_2(q-p)\gamma_2,r_2\left( pq-\gamma_1(q-p) \right),r_2pq\gamma_2  \}\subset \N,$ and $b=\lcm(\{r_1p\gamma_2,r_2pq\gamma_2\}).$ So, the semigroup given by the inequality
\begin{equation}\label{eq_from_triangle}
\frac{br_1\gamma_2}{r_1p\gamma_2}\, x_1 + \frac{b r_1(p-\gamma_1)}{r_1p\gamma_2}\,x_2 \mod b \le \frac{b r_2(q-p)\gamma_2}{r_2pq\gamma_2}\, x_1 + \frac{br_2\left( pq-\gamma_1(q-p) \right)}{r_2pq\gamma_2}\,x_2 ,
\end{equation}
satisfies Lemma \ref{lemma_triangles} for $\T_{1}=\T.$
\end{remark}

\begin{example}\label{ejemplo_triangles}
Consider $S$ the $\N^2$-semigroup showed in Figure \ref{figure_triangles}, that is, the nonnegative integer solutions of the modular inequality $11x+ 6y\mod  110\leq  3x+15y$. In this example, the vertex set of the triangle $\T_1=\T$ is $\{(0,0),(10,0),(\frac{300}{49},\frac{880}{174})\}$. The vertex sets of  $T_2$ and $T_3$ are $\{ (\frac{55}{4}, 0), (20,0), (\frac{880}{49},\frac{550}{147})\}$ and $\{ (\frac{55}{2}, 0), (30,0), (\frac{1430}{49}, \frac{220}{147}) \}$, respectively. Besides, the vectors $(\mu_1,\mu_2)$ and $(-\nu_1,\nu_2)$ determining the triangles $\T_i$ are $(\mu_{1,3},\mu_{2,3})=(\frac{9}{17},\frac{8}{17})$ and $(-\nu_{1,3},\nu_{2,3})=(-\frac{6}{17},\frac{11}{17})$, respectively.
\end{example}

\section{Testing $\N^n$-semigroups for being proportionally modular affine semigroups. Case 2.}\label{SectionCase2}

In this section, $\N^n$-semigroups containing some $e_i$ are considered. So, we assume $S$ is an $\N^n$-semigroup, any $L=\{[p_1,q_1],\ldots , [p_t,q_t]\}$ set of closed intervals satisfies that $\phi(L)=\phi([p_t,q_t])$, $S_i\neq \N$ is a proportionally modular numerical semigroup for all $i\in [t]$, and $S_i=\N$ for all $i\in \{t+1,\ldots ,n\}$. Since $S$ is a semigroup, for every $x\in \N^n\setminus S$ and $i\in \{t+1,\ldots ,n\},$  $x-e_i\notin S.$ We also consider the vectors
$\mu_{1,i}e_t+\mu_{2,i}e_i$ and $-\nu_{1,i}e_t+\nu_{2,i}e_i$ where $\mu_{1,i},\nu_{1,i}\in [0,1),$ $\mu_{2,i},\nu_{2,i}\in (0,1]$, and $\mu_{1,i}+\mu_{2,i}=\nu_{1,i}+\nu_{2,i}=1$ for $i\in \{t+1,\ldots ,n\}$. We define by $S^d$ and $S^u$ the sets $\sigma_{[t]}(S)\equiv S\cap \langle e_1,\ldots ,e_t\rangle _\R$ and $\{(\alpha_1,\ldots ,\alpha _n)\in S\mid \sum_{i=t+1}^n\alpha_i\neq 0\}$, respectively. Note that $S^d$ is an $\N^t$-semigroup and $S=(S^d\times \{0\}^{n-t+1})\cup S^u.$

In section \ref{SectionCase1}, we define several objects for a given set $L$ including $n$ closed intervals, but here $L$ only has $t$ elements (note that $n>t$). In order to not include so much notations, we consider those objects defined over the $\N^t$-semigroup $S^d.$

If $S$ is a proportionally modular $\N^n$-semigroup defined by the inequality $f(x)\mod b\leq g(x)$, above conditions mean that $b>f_i>g_i>0$ for all $i\in[t]$, and $g_i\ge f_i$ and $g_i>0$ for all $i\in\{t+1,\ldots ,n\}.$
By Lemma \ref{lemma_triangles} and Remark \ref{FromTriangleToEquation}, fixed $i\in [t]$ and $j\in \{t+1,\ldots ,n\}$, the semigroup $S\cap \langle e_i,e_j\rangle _\R$ is equivalent to an $\N^2$-semigroup determinated by a triangle $\T_{ij}$.
So, by Theorem \ref{CharactPropAffMod} and Corollary \ref{GAPS}, the hyperplanes defining $S$ are determined by the points $p_1e_1,\ldots ,p_te_t,q_1e_1,\ldots ,q_te_t$ (suppose $S_i=\CaS([p_i,q_i])$ for $i\in [t]$) and the edges of the triangles $\T_{t\,(t+1)},\ldots ,\T_{t\,n}$, that is, the hyperplanes are fixed by their intersections with the planes $\langle e_t,e_j\rangle _\R$ for any $j\in \{t+1,\ldots ,n\}$. Moreover, the hyperplane $F_i$ is given by the points $ip_1e_1,\ldots ,ip_te_t$ and the vectors $-\nu_{1,j}e_t+\nu_{2,j}e_j$, and $D_i$ by the points $iq_1e_1,\ldots ,iq_te_t$ and $\mu_{1,j}e_t+\mu_{2,j}e_j$, with $j\in\{t+1,\ldots ,n\}.$ Note that these datum are enough to determine a hyperplane in $\N^n.$

For generalizing the two dimensional case studied in section \ref{seccionTriangulos},
for any $i$ in $[\phi(L)]$, denote by $\Pr_{iL}$ the set
$\{ (\alpha_1,\ldots \alpha _t,\beta_{t+1},\ldots ,\beta_n)\in \R_\ge^n \mid \alpha\in \mathbb{H}_{i{L}}\}$,
and by
\begin{itemize}
\item $\overline{\Pr}^+_{iL}= \{\alpha\in S^u \mid \alpha-e_j \in \Pr_{iL}\setminus S \text{ for some } j\in [t]\}$,
\item $\overline{\Pr}^-_{iL}= \{\alpha\in S^u \mid \alpha+e_j \in \Pr_{iL}\setminus S \text{ for some } j\in [t]\}$,
\item $\overline{\Pr}^*_{iL}= \{\alpha \in S^u \setminus (\overline{\Pr}^+_{iL} \cup \overline{\Pr}^-_{iL})\mid \alpha-e_j\in \Pr_{iL}\setminus S \text{ for some } j\in \{t+1,\ldots ,n\}\}.$
\end{itemize}
Note that $\sigma_{[t]} (\Pr_{iL})$ is equal to $\mathbb{H}_{i{L}}$ for any $i\in [\phi(L)]$. So, three necessary conditions for $S$ to be a proportionally modular $\N^n$-semigroup given by $L=\{[p_1,q_1],\ldots [p_t,q_t]\}$ are $\N^n\setminus S\subset \cup _{i\in [\phi(L)]} \Pr_{iL}\subset \cup _{i\in [\phi(L)]} \Pr_{i\widetilde{L}}$, for all natural vector $\alpha\in \Pr_{iL}\setminus S$, $\alpha-e_j\notin S$ for every $j\in \{t+1,\ldots ,n\},$  and $(\alpha,\beta)\in S$ for all $(\alpha,\beta)\in (\cup _{i\in \N}\P_L\cap \N^{t}) \times \N^{n-t}.$

For the 3-dimensional case, Figure \ref{ejemplo3D} shows the geometrical arrangement of the case solved in this section.

\begin{theorem}\label{casoNnprisma}
Let $S$ be an $\N^n$-semigroup such that $S_i$ is a non proper proportionally modular numerical semigroup for all $i\in [t]$ and $S_i=\N$ for all $i\in \{t+1,\ldots ,n\}.$
Then, $S$ is a proportionally modular semigroup if and only if for some $([\widetilde{p}_1,\widetilde{q}_1],\ldots, [\widetilde{p}_t,\widetilde{q}_t])\in \widetilde{L}_{S_1}\times \cdots \times \widetilde{L}_{S_t}$, there exist $L=\{[p_1,q_1],\ldots, [p_t,q_t]\}$ with $p_1,\ldots , p_t,q_1, \ldots ,q_t\in \Q$,
and $\mu_{1,t+1},\mu_{2,t+1},\ldots ,\mu_{1,n},\mu_{2,n},$ $\nu_{1,t+1},\nu_{2,t+1},\ldots ,\nu_{1,n},\nu_{2,n} \in \Q$
satisfying the following conditions:
\begin{enumerate}
\item \label{condicion1}$\N^n\setminus S\subset \cup_{i\in [\phi(\widetilde{L})]}\Pr_{i\widetilde{L}}$;\label{inclusion_huecos}
\item \label{condicion2} for all $i\in [t]$ such that $[\widetilde{p}_i,\widetilde{q}_i]$ is a half-line interval, $\widehat{p}_i<p_i\le \widetilde{p}_i<\widetilde{q}_i\le q_i$, in other case, $\widehat{p}_i<p_i\le \widetilde{p}_i<\widetilde{q}_i\le q_i< \widehat{q}_i$;
\item \label{condicion3} for all $x\in \VSet((\N^t\setminus S^d)\cap \mathbb{H}_{1\widetilde{L}})$, $h_{11L}(x)>0$, and for all $i\in \{2,\ldots ,\phi(\widetilde{L})\}$ and $x\in \VSet( (\N^t\setminus S^d)\cap \mathbb{H}_{i\widetilde{L}})$, $h_{1iL}(x) >0$ and $h_{2(i-1)L}(x)<0$;

\item \label{condicion4} for every $s\in \Lambda_S\cap (S^d\times \{0\}^{n-t+1})$ such that $\theta_{\widetilde{L}}(s)=0$,    $\kappa_{\widehat{L}}(s)\neq 0$ and $h_{1iL}(\sigma_{[t]}(s)/m) \le 0$ and $h_{2iL}(\sigma_{[t]}(s)/m)\ge 0$ for some $m\in [\kappa_{\widehat{L}}(s)]$;

\item \label{condicion5} $\mu_{1,i},\nu_{1,i}\in [0,1),$ $\mu_{2,i},\nu_{2,i}\in (0,1]$, and $\mu_{1,i}+\mu_{2,i}=\nu_{1,i}+\nu_{2,i}=1$ for every $i\in  \{t+1,\ldots ,n\}$;
\item \label{condicion6} for all $i\in [\phi(\widetilde{L})]$, and for every
$\alpha\in \VSet(\{x\in \Pr_{i\widetilde{L}}\cap (\N^n\setminus S)\mid \sum_{j=t+1}^n x_j\neq 0\})$, $\beta \in \overline{\Pr}^+_{i\widetilde{L}},$ $\gamma\in \overline{\Pr}^-_{i\widetilde{L}}$ and $\delta \in \overline{\Pr}^*_{i\widetilde{L}}:$
\begin{enumerate}\label{equations}
    \item \label{condicion6_1} $\tau_{1(i-1)L}(\alpha)< 0$ and $\tau_{2iL}(\alpha)>0$,\label{eq1}% los huecos dentro
    \item \label{condicion6_2} $\tau_{1(i-1)L}(\gamma)\ge 0$,
    \item \label{condicion6_3} $\tau_{2iL}(\beta)\le 0$,
    \item \label{condicion6_4} $\tau_{1(i-1)L}(\delta)\ge 0$ and/or $\tau_{2iL}(\delta)\le 0$.
\end{enumerate}
\end{enumerate}

\end{theorem}

\begin{proof}
Assume $S$ is proportionally modular $\N^n$-semigroup such that $S_i=\CaS([p_i,q_i])$ for all $i\in [t]$ and $S_i=\N$ for all $i\in \{t+1,\ldots ,n\}.$ By Corollary \ref{corolario_caso_bandas}, the conditions \ref{condicion2}, \ref{condicion3} and \ref{condicion4} hold. Besides, since $\N^t\setminus S^d=\cup _{i\in [\phi(L)]} \sigma _{[t]}(\Pr_{iL})$, $\N^n\setminus S\subset \cup _{i\in [\phi(L)]} \Pr_{iL}\subset \cup _{i\in [\phi(L)]} \Pr_{i\widetilde{L}}$.

By Theorem \ref{CharactPropAffMod} and Corollary \ref{GAPS}, there exist two families of half-spaces $\{F_i^+\}_{i\in \N}$ and $\{D_i^-\}_{i\in \N},$ where $F_i^+\equiv f(x)\ge ib$ and $D_i^-\equiv d(x)\le ib$, such that $S= \cup _{i\in \N} (F_i^+\cap D_i^-)\cap \N^n$ and $(\N^n\setminus S)\cap G_0^+=\cup _{i\in \N^*} (F_i^-\cap D_{i-1}^+\cap G_0^+)\cap \N^n.$ Let $\T_{t\,i}$ be the triangle $(F_1^-\cap D_{0}^+\cap G_0^+)\cap \langle e_t,e_i\rangle _\R$ with $i\in \{t+1,\ldots ,n\}.$ As in Remark \ref{FromTriangleToEquation}, for each triangle $\T_{t\,i}$, we fix the vectors $-\nu_{1,i}e_t+\nu_{2,i}e_i$ and $\mu_{1,i}e_t+\mu_{2,i}e_i$ satisfying $\mu_{1,i},\nu_{1,i}\in [0,1),$ $\mu_{2,i},\nu_{2,i}\in (0,1]$, and $\mu_{1,i}+\mu_{2,i}=\nu_{1,i}+\nu_{2,i}=1$. So, $F_i$ is the hyperplane containing the points $\{ p_1e_1,\ldots ,p_te_t\}$ and the vectors $\{-\nu_{1,(t+1)}e_t+\nu_{2,(t+1)}e_{t+1}, \ldots ,-\nu_{1,n}e_t+\nu_{2,n}e_n \},$ and $D_i$ contains $\{ q_1e_1,\ldots ,q_te_t\}$ and $\{\mu_{1,(t+1)}e_t+\mu_{2,(t+1)}e_{t+1}, \ldots ,\mu_{1,n}e_t+\mu_{2,n}e_n \}$. Then, $F_i$ is equal to the hyperplane defined by $\tau_{2iL}(x)=0$ and $D_i\equiv\tau_{1iL}(x)=0$. Furthermore, since $q_1,\ldots ,q_t$, $\mu_{2,t+1},\ldots \mu_{2,n}$, $p_1,\ldots ,p_t$, $\nu_{2,t+1},\ldots ,\nu_{2,n}$ belong to $\R_>$, the closed half-space $F_i^+$ is defined by $\tau_{2iL}(x)\le 0$, and the opened half-space $F_i^-$ by $\tau_{2iL}(x)> 0$. Analogously, $D_i^-\equiv \tau_{1iL}(x)\ge 0$, and $D_i^+\equiv\tau_{1iL}(x)< 0$. Again, by Theorem \ref{CharactPropAffMod} and Corollary \ref{GAPS}, the conditions \ref{condicion6_1}, \ref{condicion6_2}, \ref{condicion6_3} and \ref{condicion6_4} hold.

Conversely, let $S'$ be the $\N^n$-semigroup $\cup _{i\in \N} (F_i^+\cap D_i^-)\cap \N^n$ where $F_i^+$ is the closed half-space defined by $\tau_{2iL}(x)\le 0$, and $D_i^-\equiv \tau_{1iL}(x)\ge 0$. By Theorem \ref{CharactPropAffMod}, $S'$ is a proportionally modular semigroup. Conditions \ref{condicion1}, \ref{condicion2}, \ref{condicion3} and \ref{condicion4} imply $S^d=S'^d$. Since $\N^n\setminus S^u\subset \N^n\setminus S'^u$ by conditions \ref{condicion1} and \ref{condicion6_1}, $S'^u \subset S^u$.

Suppose $\alpha\in S^u$.
If $\alpha$ belongs to $(\cup _{i\in \N}\P_{\widetilde{L}}\cap \N^{t}) \times \N^{n-t}$, $\alpha \in S'^u$. In other case, $\alpha$ in $\Pr_{i\widetilde{L}}$ for some $i\in [\phi(\widetilde{L})].$ If $\alpha \in \overline{\Pr}^+_{i\widetilde{L}}\cup  \overline{\Pr}^-_{i\widetilde{L}}\cup \overline{\Pr}^*_{i\widetilde{L}}$, $\alpha \in (F_{i-1}^+\cap D_{i-1}^-)\cup (F_i^+\cap D_i^-)$ (by conditions \ref{condicion6_2}, \ref{condicion6_3} and \ref{condicion6_4}). In the case, $\alpha \in \Pr_{i\widetilde{L}}\setminus (\overline{\Pr}^+_{i\widetilde{L}}\cup  \overline{\Pr}^-_{i\widetilde{L}}\cup \overline{\Pr}^*_{i\widetilde{L}})$, and since $\alpha \in \Pr_{i\widetilde{L}}$, $\sigma _{[t]} (\alpha)\notin S^d$.
So, there exists $\beta \in \Pr_{i\widetilde{L}}\cap (\N^n\setminus S^u)$ with $\sigma _{[t]} (\beta)=\sigma _{[t]} (\alpha)$ such that $\beta+ e_j\in S^u$ and $\pi_{\{t+1,\ldots ,n\}}(\alpha-\beta-e_j)\ge 0$, for some $j\in \{t+1,\ldots ,n\}$, that is, $\alpha=(\alpha-\beta- e_j) + \beta +e_j$ with $ \beta +e_j\in S^u$ but $\beta \notin S^u$; equivalently, $\tau_{1(i-1)L}(\beta)< 0$ and $\tau_{2iL}(\beta)>0$, but $\tau_{1(i-1)L}(\beta +e_j)\ge 0$ and/or $\tau_{2iL}(\beta + e_j)\le 0$. If $\tau_{1(i-1)L}(\beta +e_j)\ge 0$, it is easy to prove that $\tau_{1(i-1)L}(\alpha)\ge 0$. In a similar way, if $\tau_{2iL}(\beta + e_j)\le 0$, $\tau_{2iL}(\alpha)\le 0$. We can conclude $\alpha \in (F_{i-1}^+\cap D_{i-1}^-)\cup (F_i^+\cap D_i^-)\subset S'$.
\end{proof}

Algorithm \ref{algoritmo_NO_todos_intervalos} presents a computational method to check if an $\N^n$-semigroup is a proportionally modular semigroup by testing the conditions given in above theorem. Note that some steps in this algorithm can be computed in a parallel way.

\begin{algorithm}[h]\label{algoritmo_NO_todos_intervalos}
	\KwIn{An $\N^n$-semigroup $S$ with $e_i\notin S$, $\forall i \in [t],$ but $e_i\in S$, $\forall i \in \{t+1,\ldots ,n\},$ given by its set of gaps, and $\Lambda_{S^d}$ the minimal generating set of $S^d$.}
	\KwOut{If $S$ is a proportionally modular semigroup, the values of $(p_1,\ldots , p_t,q_1, \ldots ,q_t,\mu_{1,(t+1)},\mu_{2,t+1},\nu_{1,t+1},\nu_{2,t+1},\ldots ,\mu_{1,n},\mu_{2,n},\nu_{1,n},\nu_{2,n})$ determining the hiperplanes $F_1^+$ and $D_{0}^-$, the empty set in other case.}

\Begin{
    $L\leftarrow \{[p_1,q_1],\ldots, [p_t,q_t]\}$\;
    $M \leftarrow \cup _{i\in  \{t+1,\ldots ,n\}}\{ 1> \mu_{1,i}\ge 0,\, 1>\nu_{1,i}\ge 0,\, 1\ge \mu_{2,i}> 0,\, 1\ge \nu_{2,i}> 0,\, \mu_{1,i}+\mu_{2,i}=1,\, \nu_{1,i}+\nu_{2,i}=1\}$\;
    $\Delta\leftarrow\widetilde{L}_{S_1}\times \cdots \times \widetilde{L}_{S_t}$ (Algorithm \ref{algoritmo_intervalos_max_m_yin})\;

    \While{$\Delta\neq \emptyset$}
        {
        $\widetilde{L}\leftarrow \mbox{First}(\Delta)$\;
        $\{s_1,\ldots ,s_k\}\leftarrow \{s\mid s\in \Lambda_{S^d} \text{ and } \theta_{\widetilde{L}}(s)=0\}$\;

        \If{$\prod _{j\in [k]}\kappa_{\widehat{L}}(s_j)\neq 0$ and $\N^n\setminus S\subset \cup _{i\in [\phi(\widetilde{L})]} \Pr_{i\widetilde{L}}$
        }
            {

            $E\leftarrow \{h_{11L}(x)>0\mid x\in \VSet((\N^t\setminus S^d)\cap \mathbb{H}_{1\widetilde{L}})\}$\;
            $E\leftarrow E\cup \{\{h_{1iL}(x) >0, h_{2(i-1)L}(x)<0\}\mid i\in \{2,\ldots ,\phi(\widetilde{L})\}\mbox{ and } x\in \VSet((\N^t\setminus S^d)\cap \mathbb{H}_{i\widetilde{L}})\}$\;

            $E\leftarrow E\cup \{\{\tau_{1(i-1)L}(\alpha)< 0, \tau_{2iL}(\alpha)>0\}\mid  i\in [\phi(\widetilde{L})] \mbox{ and } \alpha\in \VSet(\Pr_{i\widetilde{L}}\cap \{x\in\N^n\setminus S\mid \sum _{j=t+1}^n x_j\neq 0\})\}$\;

            $F \leftarrow \{\tau_{2iL}(\beta)\le 0 \mid  i\in [\phi(\widetilde{L})] \mbox{ and } \beta \in \overline{\Pr}^+_{i\widetilde{L}}\}$\;

            $F \leftarrow F\cup \{\tau_{1(i-1)L}(\gamma)\ge 0 \mid  i\in [\phi(\widetilde{L})]\mbox{ and } \gamma\in \overline{\Pr}^-_{i\widetilde{L}}\}$\;

            $\Gamma \leftarrow \Cross _{i\in [\phi(\widetilde{L})],\, \delta\in \overline{\Pr}^*_{i\widetilde{L}}} \{\tau_{1(i-1)L}(\delta)\ge 0,\tau_{2iL}(\delta)\le 0 \}$\;

            $\Omega\leftarrow [\kappa_{\widehat{L}}(s_1)]\times \cdots \times [\kappa_{\widehat{L}}(s_k)]$\;

            \While{$\Omega\neq\emptyset$}
                {
                $(m_1,\ldots ,m_k)\leftarrow \mbox{First}(\Omega)$\;

                $F'\leftarrow  F\cup \{\{h_{11L}(s_i/m_i) \le 0, h_{21L}(s_i/m_i)\ge 0\}\mid i\in [k]\} $\;

                $\Gamma'\leftarrow \Gamma$\;

                \While{$\Gamma'\neq \emptyset$}
                    {
                    $F^* \leftarrow \mbox{First}(\Gamma')$\;
                    $T \leftarrow \mbox{ Solve } \big(\cup_{i\in [t]}\{r_{[\widetilde{p}_i,\widetilde{q}_i]}\}\bigcup M
                    \bigcup E \bigcup F' \bigcup F^*\big)$ for $\{p_1,\ldots , p_t,q_1, \ldots ,q_t,\mu_{1,t+1},\mu_{2,t+1},\nu_{1,t+1},\nu_{2,t+1},\ldots ,\mu_{1,n},\mu_{2,n},\nu_{1,n},\nu_{2,n}\}$
                    \;
                    \If{$(p_1,\ldots , p_t,q_1, \ldots ,q_t,\mu_{1,t+1},\mu_{2,t+1},\nu_{1,t+1},\nu_{2,t+1},\ldots ,\mu_{1,n},\mu_{2,n},\nu_{1,n},\nu_{2,n})\in T\cap \R^{2t+4(n-t)}$}
                        {\Return $(p_1,\ldots , p_t,q_1, \ldots ,q_t,\mu_{1,t+1},\mu_{2,t+1},\nu_{1,t+1},\nu_{2,t+1},\ldots ,\mu_{1,n},\mu_{2,n},\nu_{1,n},\nu_{2,n})$}
                    $\Gamma'\leftarrow \Gamma'\setminus \{F^*\}$\;
                    }
                $\Omega\leftarrow \Omega\setminus \{(m_1,\ldots ,m_k)\}$\;
                }
        }
        $\Delta\leftarrow \Delta\setminus \{\widetilde{L}\}$\;
        }
	\Return $\emptyset$\;
}
\caption{Checking if an $\N^n$-semigroup $S$ with $e_i\notin S$, $\forall i \in [t],$ $e_i\in S$, $\forall i \in \{t+1,\ldots ,n\},$ and $S_1,\ldots ,S_t$ proportionally modular numerical semigroups, is a proportionally modular semigroup.}
\end{algorithm}

\begin{example}
Let $S$ be the $\N^3$-semigroup which gap set is the set of black points in Figure \ref{ejemplo3D}, that is,
\[
\begin{multlined}
\{
(0, 1, 0), (0, 2, 0), (0, 2, 1), (0, 5, 0), (1, 0, 0), (1, 2, 0), (1,3, 0), (1, 6, 0), (2, 0, 0), (2, 0, 1),\\ (2, 3, 0), (3, 0, 0), (3, 1, 0), (3, 4, 0), (4, 1, 0)
\}
\end{multlined}
\]
So, the $\N^2$-semigroup $S\cap \langle e_1,e_2\rangle _\R$ is minimally generated by
$$
\{
(0, 3), (0, 4), (1, 1), (2, 1), (4, 0), (5, 0), (5, 2),(6, 0), (7, 0)
\}.
$$
\begin{figure}[h]
  \centering
\includegraphics[scale=.27]{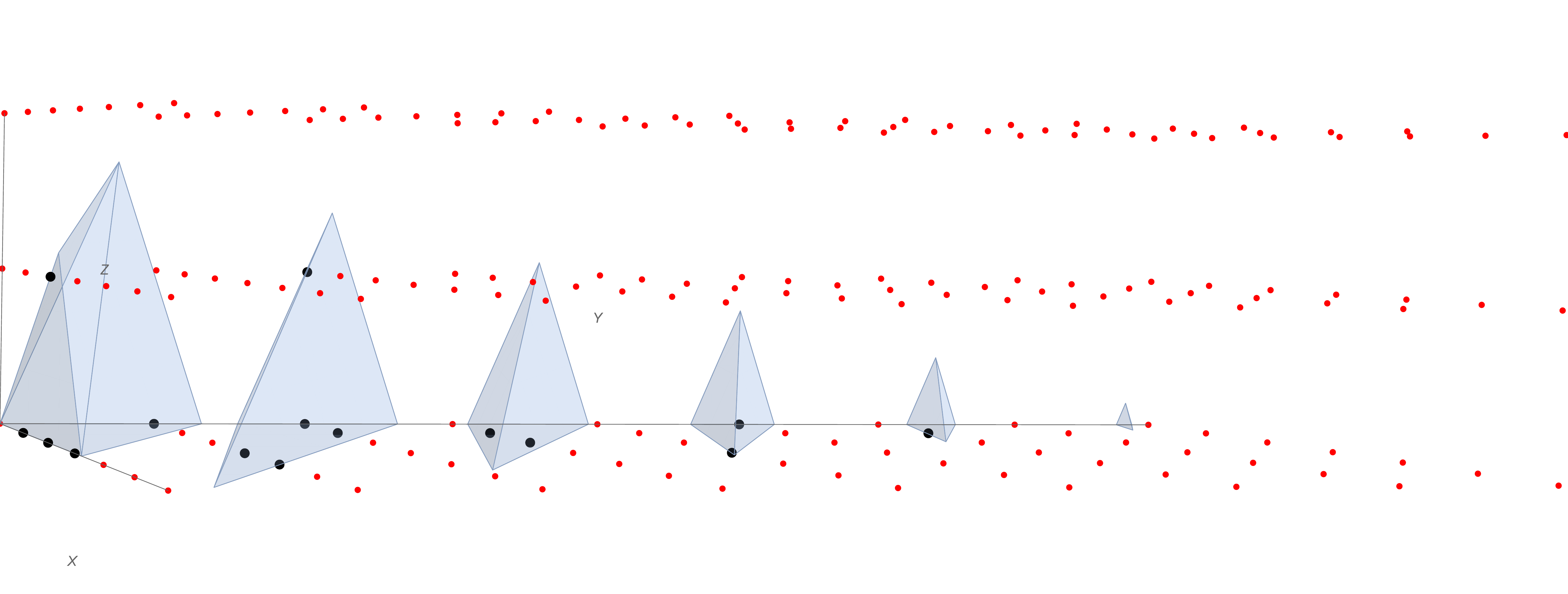}
  \caption{$\N^3$-semigroup given by
  $29x+ 11y + 6z\mod  33\leq  6x+3y+15z$.}\label{ejemplo3D}
\end{figure}
By Algorithm \ref{algoritmo_NO_todos_intervalos}, the semigroup $S$ is a proportionally modular affine semigroup where the intervals $[\frac{829}{256},\frac{113}{16}]$ and $[\frac{21}{16},\frac{1589}{1024}]$ determines $S_1$ and $S_2$ respectively, and $(\mu_{1,3},\mu_{2,3})=(\frac{39}{128},\frac{89}{128})$ and $(-\nu_{1,3},\nu_{2,3})=(-\frac{1}{4},\frac{3}{4})$.

Using our software \cite{prog_prop_mod_afines}, the above results are obtained,
\begin{verbatim}
In[1]:= mgs = {{0, 3}, {0, 4}, {1, 1}, {2, 1}, {4, 0}, {5, 0}, {5, 2},
       {6, 0}, {7, 0}};
In[2]:= gaps = {{0, 1, 0}, {0, 2, 0}, {0, 2, 1}, {0, 5, 0}, {1, 0, 0}, {1, 2, 0},
        {1, 3, 0}, {1, 6, 0}, {2, 0, 0}, {2, 0, 1}, {2, 3, 0}, {3, 0, 0},
        {3, 1, 0}, {3, 4, 0}, {4, 1, 0}};

In[3]:= IsNnProportionallyModularSemigroup[mgs, gaps]

Out[3]= {829/256, 21/16, 113/16, 1589/1024, 39/128, 89/128, 1/4, 3/4}

\end{verbatim}

\end{example}

\subsubsection*{Acknowledgements}
The authors were partially supported by Junta de Andaluc\'{\i}a research group FQM-366.
The first author was supported by Programa Operativo de Empleo Juvenil 2014-2020, Fondo Social Europeo, (GJLIC114).
The second, third and fourth authors were partially supported by the project MTM2017-84890-P (MINECO/FEDER, UE), and the fourth author was partially supported by the project MTM2015-65764-C3-1-P (MINECO/FEDER, UE).

\end{document}